\documentclass[11pt]{amsart}

\usepackage[T1]{fontenc}
\usepackage{lmodern}
\usepackage{amsmath,amssymb,amsthm}
\usepackage{microtype}
\usepackage{tikz}
\usepackage[
  colorlinks=true,
  linkcolor=blue,
  citecolor=blue,
  urlcolor=blue,
  pdfusetitle
]{hyperref}

\newtheorem{theorem}{Theorem}
\newtheorem{lemma}[theorem]{Lemma}

\newcommand{\dcen}{\delta_{\mathrm{cen}}}
\newcommand{\Cov}{\operatorname{Cov}}
\newcommand{\Rtwo}{\mathbb R^2}
\newcommand{\Btwo}{B^2}

\title[Breaking the Fourth Wall]
{Breaking the Fourth Wall: The Planar Centroid Banach--Mazur Diameter
Is Strictly Less than \(4\)}

\author[Scott Duke Kominers]{Scott Duke Kominers}
\address{Harvard Business School; Department of Economics and Center of
Mathematical Sciences and Applications, Harvard University; and a16z crypto}
\email{kominers@fas.harvard.edu}

\thanks{I appreciate helpful comments and encouragement from Marek Lassak.
I used LLMs to assist with computations, analysis, and synthesis in the
preparation of this article, especially GPT-5.6~Sol and Claude~Fable~5
(both accessed in part via Poe with the support of Quora, where I am an
advisor). In particular, GPT-5.6~Sol identified an analytic shortcut that
substantially simplified a prior version of the argument. The problem, final
methods, and final written form are my own; and of course any errors remain
my responsibility. This work was conducted while I was visiting the
Technological Innovation, Entrepreneurship, and Strategic Management
(TIES) Group at the MIT Sloan School of Management; I greatly appreciate
their hospitality.}

\begin{document}

\begin{abstract}
We prove that the centroid Banach--Mazur diameter of planar convex
bodies is strictly less than \(4\), (slightly) improving the prior upper
bound \(69/17\approx 4.05882\) due to Lassak. The proof combines a sharp
covariance-ellipse sandwich, whose equality cases in any single direction
occur only for triangles, with a compactness argument: an extremal pair at
distance \(4\) would consist of two triangles, which are linearly equivalent
and hence must in fact be at distance \(1\). The resulting gap below \(4\)
is uniform but nonquantitative.
\end{abstract}

\maketitle

\section{Introduction}

A \emph{convex body} is a compact convex \(K\subseteq\Rtwo\) with
nonempty interior; we write \(c(K)\) for its centroid and
\(K_0=K-c(K)\). The \emph{centroid Banach--Mazur distance} of two
planar convex bodies is
\[
  \dcen(C,D)
  =
  \inf\bigl\{t\ge1:
  C_0\subseteq A D_0\subseteq tC_0
  \text{ for some }A\in \mathrm{GL}(2)\bigr\}.
\]
The distance \(\dcen\) is symmetric, affine invariant in each argument,
and satisfies the multiplicative triangle inequality; each assertion
follows by composing the corresponding linear maps and inclusions after
centering.

The \emph{planar centroid Banach--Mazur diameter} is the supremum of
\(\dcen\) over all pairs of planar convex bodies; our main result implies
that this diameter is strictly less than \(4\).

\begin{theorem}\label{thm:main}
There is an absolute constant \(\varepsilon_0>0\) such that
\[
  \dcen(C,D)\le4-\varepsilon_0
\]
for all planar convex bodies \(C\) and \(D\).
\end{theorem}

Gr\"unbaum introduced the centroid-constrained version of the
Banach--Mazur distance and conjectured that its planar diameter is
\(5/2\), attained by a parallelogram and a triangle
\cite[pp.~259--260]{Gr63}. Lassak proved that the distance of that
pair is indeed \(5/2\) \cite{Las24a}, and also established the
general upper bound
\[
  \dcen(C,D)\le\frac{69}{17}\approx 4.05882
\]
\cite{Las24b}. To our knowledge, Theorem~\ref{thm:main} is the first
proof that the diameter is below \(4\).

The proof is driven by a sharp one-body statement. Associated to every
planar convex body with centroid at the origin is a canonical covariance
ellipse whose support distance in each direction is proportional to the
root-mean-square spread of the uniform mass in that direction. With the
normalization used below, the ellipse lies in the body and its double
contains the body. Contact with the doubled ellipse in even one direction
forces the body to be a triangle. After independently normalizing two
bodies so that their covariance ellipses are both the Euclidean unit disk,
chaining these inclusions gives the coarse distance bound \(4\).

The normalized class is compact and \(\dcen\) is continuous on it, so
the diameter is attained. If that diameter were \(4\), equality would
be required in the one-body estimate for both members of a maximizing
pair; both would therefore be triangles. But any two triangles with
centroid at the origin are linearly equivalent and have distance \(1\),
giving the desired contradiction.

This also explains where the improvement over Lassak's bound enters.
Lassak's argument uses affine-regular hexagons supplied by
Besicovitch's theorem \cite{Bes48}, together with the sharp \(4/21\)
bound on the displacement of the centroid from the hexagon center
established in \cite{Las24c}, leading to \(69/17\) \cite{Las24b}. Our
covariance ellipses are built from centroid-centered second moments,
giving the coarse bound \(4=2\cdot2\). Thus the two sharp one-body
factors create the apparent ``fourth wall'' \(2\cdot2=4\); it is
broken not by improving either factor, but because simultaneous equality
forces two affinely-equivalent triangles, which break the hypothesized
equality.

\section{A one-dimensional moment window}

The geometric argument reduces each direction to a nonnegative
concave function with vanishing first moment
(Figure~\ref{fig:tents}(a)). We first isolate the resulting
one-dimensional estimate.

For a nonzero integrable function \(w\ge0\) on a bounded interval,
we write \(m_2(w)=\int \tau^2w(\tau)\,d\tau\big/\!\int w(\tau)\,d\tau.\)

\begin{theorem}[Moment window]\label{thm:moment}
Let \(a,b>0\), and let
\(w:[-a,b]\to[0,\infty)\) be continuous, concave, and nonzero.
If
\[
  \int_{-a}^b \tau w(\tau)\,d\tau=0,
\]
then we have
\[
  \frac{b^2}{8}\le m_2(w)\le\frac{b^2}{2}.
\]
The upper and lower equality cases are, respectively,
\(a=2b,\ w(\tau)\propto\tau+2b\), and
\(a=b/2,\ w(\tau)\propto b-\tau\).
\end{theorem}

We use the following standard decomposition.

\begin{lemma}[Tent decomposition]\label{lem:tents}
Let \(a>0\). Every continuous nonnegative concave function \(w\) on
\([-a,1]\) can be written
\[
  w(\tau)
  =
  \lambda(\tau+a)+\mu(1-\tau)
  +\int_{(-a,1)}T_p(\tau)\,d\nu(p),
  \qquad \lambda,\mu\ge0,
\]
where \(\nu\) is a finite positive Borel measure and
\[
  T_p(\tau)
  =
  \min\left\{
    \frac{\tau+a}{p+a},
    \frac{1-\tau}{1-p}
  \right\},
  \qquad -a<p<1.
\]
\end{lemma}

The name is literal: the graph of \(T_p\) rises linearly from \(0\) at
\(-a\) to height \(1\) at \(p\), then descends linearly to \(0\) at
\(1\), so it has the shape of a tent with peak \(p\); see
Figure~\ref{fig:tents}(b). The two affine terms are the one-sided
wedges determined by the endpoint values of \(w\). Thus, after removing
those wedges, every concave profile is a nonnegative superposition of
tents; the form we present makes the equality analysis transparent.

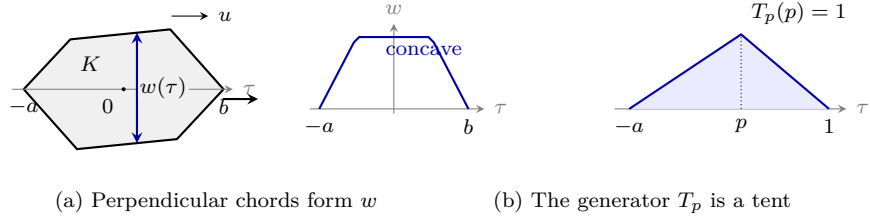
\begin{figure}[!ht]
\centering
\begin{tikzpicture}[
  x=.88cm,
  y=.88cm,
  >=stealth,
  every node/.style={font=\scriptsize},
  profile/.style={blue!65!black,thick}
]
\begin{scope}[xshift=1.55cm,yshift=1.08cm]
  \fill[gray!12]
    (-1.50,0)--(-.80,.75)--(.70,.90)--
    (1.50,0)--(.80,-.75)--(-.70,-.90)--cycle;
  \draw[thick]
    (-1.50,0)--(-.80,.75)--(.70,.90)--
    (1.50,0)--(.80,-.75)--(-.70,-.90)--cycle;
  \draw[->,gray] (-1.68,0)--(1.64,0) node[right] {\(\tau\)};
  \fill (0,0) circle (.75pt);
  \node[below left] at (0,0) {\(0\)};
  \draw[<->,profile] (.20,-.81)--(.20,.85)
    node[midway,right,fill=gray!12,inner sep=.7pt,text=black]
    {\(w(\tau)\)};
  \node[below] at (-1.50,0) {\(-a\)};
  \node[below] at (1.50,0) {\(b\)};
  \node at (-.50,.39) {\(K\)};
  \draw[->] (.72,1.10)--(1.27,1.10) node[right] {\(u\)};
\end{scope}

\draw[->,thick] (3.25,1.08)--(3.78,1.08);

\begin{scope}[xshift=5.12cm,yshift=.82cm]
  \draw[->,gray] (-1.26,0)--(1.36,0) node[right] {\(\tau\)};
  \draw[->,gray] (0,-.07)--(0,1.28) node[above] {\(w\)};
  \draw[profile]
    (-1.12,0)--(-.60,1.00)--(-.52,1.08)--
    (.52,1.08)--(.60,1.00)--(1.12,0);
  \node[below] at (-1.12,0) {\(-a\)};
  \node[below] at (1.12,0) {\(b\)};
  \node[blue!65!black] at (.45,.88) {\textnormal{concave}};
\end{scope}
\node at (3.15,-.47)
  {\textnormal{(a) Perpendicular chords form \(w\)}};

\begin{scope}[xshift=9.55cm,yshift=.82cm]
  \draw[->,gray] (-1.72,0)--(1.76,0) node[right] {\(\tau\)};
  \fill[blue!8] (-1.50,0)--(.18,1.12)--(1.50,0)--cycle;
  \draw[profile] (-1.50,0)--(.18,1.12)--(1.50,0);
  \draw[densely dotted] (.18,0)--(.18,1.12);
  \node[below] at (-1.50,0) {\(-a\)};
  \node[below] at (.18,0) {\(p\)};
  \node[below] at (1.50,0) {\(1\)};
  \node[above right] at (.18,1.12) {\(T_p(p)=1\)};
\end{scope}
\node at (9.55,-.47)
  {\textnormal{(b) The generator \(T_p\) is a tent}};
\end{tikzpicture}
\caption{(a)~The chord reduction: for fixed \(u\), the section
\(\{x:\langle x,u\rangle=\tau\}\) is perpendicular to \(u\), and its
length is \(w(\tau)\); Brunn's theorem makes \(w\) concave.
(b)~The tent generator, drawn with the normalization \(b=1\): the
tent \(T_p\) vanishes at \(-a\) and \(1\) and has height \(1\) at
\(p\). The two endpoint wedges together with a positive
superposition of tents reconstruct \(w\) (Lemma~\ref{lem:tents}).}
\label{fig:tents}
\end{figure}

\begin{proof}[Proof of Lemma~\ref{lem:tents}]
Subtract the affine chord joining the endpoint values:
\[
  \lambda=\frac{w(1)}{a+1},
  \qquad
  \mu=\frac{w(-a)}{a+1},
  \qquad
  g(\tau)=w(\tau)-\lambda(\tau+a)-\mu(1-\tau).
\]
Then \(g\) is nonnegative and concave, and vanishes at both endpoints.
Its distributional curvature
\(\sigma=-D^2g\) is a positive locally finite measure on \((-a,1)\).
Put
\[
  \gamma(p)=\frac{(p+a)(1-p)}{a+1},
  \qquad
  d\nu(p)=\gamma(p)\,d\sigma(p).
\]
Although \(\sigma\) need not be finite, \(\nu\) is. Indeed,
\(\gamma''=-2/(a+1)\), so distributional integration by parts gives
\[
  \nu((-a,1))
  =
  \int_{(-a,1)}\gamma\,d\sigma
  =
  \frac{2}{a+1}\int_{-a}^1g(\tau)\,d\tau
  <\infty.
\]
To justify the integration by parts, we may first use cutoffs a
positive distance from the endpoints. Concavity and the endpoint
values of \(g\) imply
\[
  (\tau+a)|g'_\pm(\tau)|\longrightarrow0
  \quad(\tau\downarrow-a),
  \qquad
  (1-\tau)|g'_\pm(\tau)|\longrightarrow0
  \quad(\tau\uparrow1),
\]
where \(g'_\pm\) denote the one-sided derivatives. Moreover,
\(g(\tau)\to0\). Since \(\gamma\) vanishes linearly and \(\gamma'\)
is bounded, all boundary terms tend to \(0\).

The Dirichlet Green kernel for \(-d^2/d\tau^2\) on \([-a,1]\) is
\[
  G(\tau,p)
  =
  \begin{cases}
    \dfrac{(1-p)(\tau+a)}{a+1},&\tau\le p,\\[5pt]
    \dfrac{(p+a)(1-\tau)}{a+1},&\tau\ge p,
  \end{cases}
  \ =\gamma(p)T_p(\tau).
\]
The function
\[
  H(\tau)
  =
  \int_{(-a,1)}G(\tau,p)\,d\sigma(p)
  =
  \int_{(-a,1)}T_p(\tau)\,d\nu(p)
\]
is finite and continuous by dominated convergence, and vanishes at
the two endpoints by construction. Testing against a compactly
supported smooth function and using Fubini's theorem gives
\(-D^2H=\sigma=-D^2g\). Hence \(g-H\) is affine and vanishes at both
endpoints, so \(g=H\), as required.
\end{proof}

Analytically, Lemma~\ref{lem:tents} is the one-dimensional
Green-kernel representation of a concave function in terms of its
distributional curvature. Its use below under a linear moment
constraint is closely related to geometric localization and the
reduction to affine one-dimensional densities \cite{KLS95,FG04}.

\begin{proof}[Proof of Theorem~\ref{thm:moment}]
By scaling if necessary, we may take \(b=1\).
We use the decomposition in Lemma~\ref{lem:tents}. The first moments
of its three types of generators are
\begin{align}
  \int_{-a}^1\tau(\tau+a)\,d\tau
    &=\frac{(2-a)(a+1)^2}{6},
    \label{eq:first-left}\\
  \int_{-a}^1\tau(1-\tau)\,d\tau
    &=-\frac{(2a-1)(a+1)^2}{6},
    \label{eq:first-right}\\
  \int_{-a}^1\tau T_p(\tau)\,d\tau
    &=\frac{(a+1)(1+p-a)}{6}.
\end{align}
If \(a>2\), then all three quantities are negative; if \(a<1/2\),
all three are positive. Since the coefficients and the measure in
the decomposition are nonnegative and not all zero, the mean-zero
condition \(\int \tau w(\tau)\,d\tau=0\) forces
\[
  \frac12\le a\le2.
\]

For functions whose product is integrable, write
\[
  \mathcal J(f,g)
  =
  \int_{-a}^1f(\tau)g(\tau)\,d\tau.
\]
The decomposition may be paired term by term with the polynomials
below because \(\nu\) (as given in Lemma~\ref{lem:tents}) is finite
and \(0\le T_p\le1\).

For the upper bound, set
\[
  \varphi(\tau)=\frac12-\tau^2-\frac a2\tau.
\]
Direct integration gives
\begin{align*}
  \mathcal J(\varphi,\tau+a)&=0,\\
  \mathcal J(\varphi,1-\tau)
    &=\frac{(2-a)(a+1)^3}{12},\\
  \mathcal J(\varphi,T_p)
    &=\frac{(a+1)(1-p)(p+2)}{12}.
\end{align*}
These quantities are nonnegative for
\(1/2\le a\le2\) and \(-a<p<1\). Hence
\[
  0\le\mathcal J(\varphi,w)
  =
  \frac12\int_{-a}^1w
  -
  \int_{-a}^1\tau^2w,
\]
where the first moment of \(w\) has been used. This proves that
\(m_2(w)\le1/2\).

For the lower bound, set
\[
  \psi(\tau)=\tau^2-\frac18+\frac a2\tau.
\]
The corresponding pairings are
\begin{align*}
  \mathcal J(\psi,\tau+a)
    &=\frac{3(a+1)^2}{16},\\
  \mathcal J(\psi,1-\tau)
    &=\frac{(a+1)^2(2a-1)^2}{48},\\
  \mathcal J(\psi,T_p)
    &=\frac{(a+1)(2p+1)^2}{48};
\end{align*}
these are all nonnegative, and therefore
\[
  0\le\mathcal J(\psi,w)
  =
  \int_{-a}^1\tau^2w
  -
  \frac18\int_{-a}^1w.
\]
Thus we see that \(m_2(w)\ge1/8\).

It remains to inspect equality. Equality in the upper bound eliminates
the tent measure, since its pairing is strictly positive for every
\(-a<p<1\). If \(a<2\), then equality also eliminates the
\((1-\tau)\)-component, after which the mean-zero condition eliminates
the remaining component. Hence \(a=2\); equations
\eqref{eq:first-left}--\eqref{eq:first-right} then force
\(\mu=0\), so \(w\) is proportional to \(\tau+2\).

Equality in the lower bound forces \(\lambda=0\), forces \(\mu=0\)
unless \(a=1/2\), and permits tent mass only at \(p=-1/2\). If
\(a>1/2\), this tent has strictly negative first moment, so the
mean-zero condition forces \(\nu=0\), contradicting \(w\ne0\).
Hence \(a=1/2\). The point \(p=-1/2\) is then an excluded endpoint,
so \(\nu=0\), and \(w\) is proportional to \(1-\tau\). Rescaling
from \(b=1\) proves the theorem.
\end{proof}

\section{The covariance ellipse and its equality cases}

Let \(\Btwo\) denote the Euclidean unit disk and let
\(S^1=\partial\Btwo\). The support function of a compact convex set
\(K\) is
\[
  h_K(u)=\max_{x\in K}\langle x,u\rangle
  \qquad(u\in S^1).
\]

For a convex body \(K\), put
\[
  \Cov_K
  =
  \frac1{|K|}
  \int_K
  (x-c(K))(x-c(K))^{\!\top}\,dx,
\]
where \(|K|\) denotes planar area. Since \(K\) has nonempty interior,
\(\Cov_K\) is positive definite. For \(u\in S^1\), the scalar
\(u^{\!\top}\Cov_Ku\) is the mean squared displacement in direction
\(u\) of a uniformly random point of \(K\). Define the
\emph{covariance ellipse}, with the planar normalization used here, by
\[
  E_K=\sqrt2\,\Cov_K^{1/2}\Btwo,
\]
so its support distance in direction \(u\) is
\(\sqrt{2\,u^{\!\top}\Cov_Ku}\), or \(\sqrt2\) times the corresponding
root-mean-square spread. Up to normalization, \(E_K\) is also sometimes
called the \emph{inertia ellipse}. Here \(\Cov_K^{1/2}\) is the unique
positive-definite square root; we use fractional powers for matrix
square roots and radicals for scalars.

For \(A\in\mathrm{GL}(2)\) and \(z\in\Rtwo\),
\[
  \Cov_{AK+z}=A\Cov_KA^{\!\top},
  \qquad
  E_{AK+z}=AE_K,
\]
where the second equality is equality of sets: both sides are centered
ellipses with shape matrix \(2A\Cov_KA^{\!\top}\). Thus the
construction is translation invariant and linearly equivariant.

For a convex body \(K\) with \(c(K)=0\) and \(u\in S^1\), let
\(w_{K,u}(\tau)\) be the length of the chord
\[
  K\cap\{x:\langle x,u\rangle=\tau\},
  \qquad
  -h_K(-u)\le\tau\le h_K(u),
\]
as in Figure~\ref{fig:tents}(a).

\begin{theorem}[Covariance-ellipse sandwich]\label{thm:sandwich}
Let \(K\subset\Rtwo\) be a convex body with \(c(K)=0\). Then, for
every \(u\in S^1\),
\[
  \sqrt{2\,u^{\!\top}\Cov_Ku}
  \le h_K(u)
  \le\sqrt{8\,u^{\!\top}\Cov_Ku}.
\]
Equivalently,
\[
  E_K\subseteq K\subseteq2E_K.
\]
Equality in the lower support bound occurs exactly when
\(w_{K,u}(\tau)\propto\tau+2h_K(u)\) on
\([-2h_K(u),h_K(u)]\). Equality in the upper support bound occurs
exactly when \(w_{K,u}(\tau)\propto h_K(u)-\tau\) on
\([-h_K(u)/2,h_K(u)]\).

Either equality in one direction forces \(K\) to be a triangle.
Conversely, the two orientations normal to each side of a triangle
realize the two cases.
\end{theorem}

\begin{proof}
Fix \(u\in S^1\), and put
\[
  b=h_K(u),
  \qquad
  a=h_K(-u).
\]
The centroid lies in the interior of \(K\), so \(a,b>0\). Indeed, if
a supporting linear functional vanished at the centroid, it would be
nonnegative on \(K\) and positive on a set of positive area, contrary
to its having mean zero.

Write \(w=w_{K,u}\). Brunn's theorem makes \(w\) concave. It is
continuous on the closed support interval; at the endpoints this
follows, for example, by writing the upper and lower boundaries as a
concave and a convex function and using the closedness of \(K\).
Fubini's theorem and the definitions of the centroid and covariance
give
\[
  \int_{-a}^b\tau w(\tau)\,d\tau=0,
  \qquad
  m_2(w)=u^{\!\top}\Cov_Ku,
  \qquad
  b=h_K(u).
\]
Thus \(w\) records the one-dimensional area distribution in direction
\(u\), while \(E_K\) records its root-mean-square scale. Substituting
the preceding identities into Theorem~\ref{thm:moment} gives
\[
  \frac{h_K(u)^2}{8}
  \le u^{\!\top}\Cov_Ku
  \le\frac{h_K(u)^2}{2}.
\]
Solving the right-hand inequality for \(h_K(u)\) gives the lower
support bound, while solving the left-hand inequality gives the upper
support bound:
\[
  \sqrt{2\,u^{\!\top}\Cov_Ku}
  \le h_K(u)
  \le\sqrt{8\,u^{\!\top}\Cov_Ku}.
\]
They are \emph{directionwise} inequalities because they compare the
supports separately for each choice of \(u\).

For a linear map \(A\), the image ellipse \(A\Btwo\) has support
function \(h_{A\Btwo}(u)=\lVert A^{\!\top}u\rVert.\) Since
\(\Cov_K^{1/2}\) is symmetric, we see that
\[
  h_{E_K}(u)=\sqrt{2\,u^{\!\top}\Cov_Ku},
  \qquad
  h_{2E_K}(u)=\sqrt{8\,u^{\!\top}\Cov_Ku}.
\]
For compact convex sets \(E,F\), the standard support-function
criterion says that \(E\subseteq F\) exactly when
\(h_E(v)\le h_F(v)\) for every \(v\in S^1\); this follows from their
representations as intersections of supporting half-planes. The
ellipse inclusions follow.

The one-dimensional equality statements follow directly from
Theorem~\ref{thm:moment}. To identify the body attaining equality, we
rotate coordinates so that \(u=(1,0)\), and write
\[
  K
  =
  \{(\tau,z):
    -a\le\tau\le b,\ 
    \ell(\tau)\le z\le r(\tau)\},
\]
where \(\ell\) is convex, \(r\) is concave, and \(w=r-\ell\).
In either equality case, \(w\) is affine and vanishes at one endpoint.
Since \(r=\ell+w\) is concave and \(w\) is affine, \(\ell=r-w\) is
also concave. Thus \(\ell\) is both convex and concave, so both
\(\ell\) and \(r\) are affine. Their graphs meet at the endpoint
where \(w=0\), and \(K\) is the convex hull of that point and the
opposite chord: it is a triangle (see Figure~\ref{fig:mechanism}).
Conversely, the chord profile of a triangle is affine in a direction
normal to any side, and reversing that direction interchanges the two
equality types.
\end{proof}

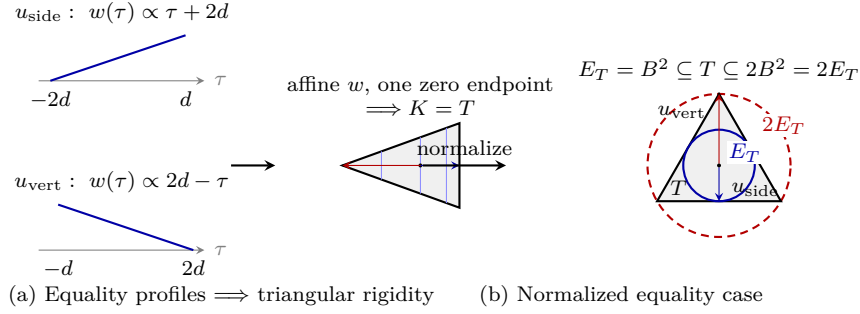
\begin{figure}[!ht]
\centering
\begin{tikzpicture}[
  x=.86cm,
  y=.86cm,
  >=stealth,
  every node/.style={font=\scriptsize},
  profile/.style={blue!65!black,thick},
  innerellipse/.style={blue!65!black,thick},
  outerellipse/.style={red!70!black,densely dashed,thick}
]
\begin{scope}[xshift=1.35cm,yshift=1.12cm]
  \draw[->,gray] (-1.28,0)--(1.27,0) node[right] {\(\tau\)};
  \draw[profile] (-1.10,0)--(.98,.70);
  \node[below] at (-1.10,0) {\(-2d\)};
  \node[below] at (.98,0) {\(d\)};
  \node[above] at (0,.73)
    {\(u_{\mathrm{side}}:\ w(\tau)\propto\tau+2d\)};
\end{scope}

\begin{scope}[xshift=1.35cm,yshift=-1.12cm]
  \draw[->,gray] (-1.28,0)--(1.27,0) node[right] {\(\tau\)};
  \draw[profile] (-.98,.70)--(1.10,0);
  \node[below] at (-.98,0) {\(-d\)};
  \node[below] at (1.10,0) {\(2d\)};
  \node[above] at (0,.73)
    {\(u_{\mathrm{vert}}:\ w(\tau)\propto2d-\tau\)};
\end{scope}

\draw[->,thick] (3.25,0)--(3.93,0);

\begin{scope}[xshift=5.30cm]
  \fill[gray!12] (-1.20,0)--(.60,-.65)--(.60,.65)--cycle;
  \draw[thick] (-1.20,0)--(.60,-.65)--(.60,.65)--cycle;
  \draw[blue!45] (-.60,-.217)--(-.60,.217);
  \draw[blue!45] (0,-.433)--(0,.433);
  \draw[blue!45] (.40,-.578)--(.40,.578);
  \fill (0,0) circle (.75pt);
  \draw[->,blue!65!black] (0,0)--(.60,0);
  \draw[->,red!70!black] (0,0)--(-1.20,0);
  \node[align=center] at (0,1.08)
    {\(\textnormal{affine \(w\), one zero endpoint}\)\\[-1pt]
     \(\Longrightarrow K=T\)};
\end{scope}

\draw[->,thick] (6.20,0)--(7.48,0)
  node[midway,above] {\textnormal{normalize}};

\begin{scope}[xshift=9.25cm]
  \node at (0,1.48)
    {\(E_T=\Btwo\subseteq T\subseteq2\Btwo=2E_T\)};
  \draw[outerellipse] (0,0) circle (1.10);
  \fill[gray!12]
    (0,1.10)--(-.95263,-.55)--(.95263,-.55)--cycle;
  \draw[thick]
    (0,1.10)--(-.95263,-.55)--(.95263,-.55)--cycle;
  \draw[innerellipse] (0,0) circle (.55);
  \fill (0,0) circle (.75pt);
  \draw[->,red!70!black] (0,.03)--(0,1.10);
  \node[left] at (-.04,.82) {\(u_{\mathrm{vert}}\)};
  \draw[->,blue!65!black] (0,-.03)--(0,-.55);
  \node[right] at (.04,-.37) {\(u_{\mathrm{side}}\)};
  \node[blue!65!black,fill=white,inner sep=.8pt] at (.40,.20)
    {\(E_T\)};
  \node[red!70!black,fill=white,inner sep=.8pt] at (1.00,.62)
    {\(2E_T\)};
  \node at (-.62,-.34) {\(T\)};
\end{scope}

\node at (3.10,-2.00)
  {\textnormal{(a) Equality profiles \(\Longrightarrow\)
  triangular rigidity}};
\node at (9.25,-2.00)
  {\textnormal{(b) Normalized equality case}};
\end{tikzpicture}
\caption{The directionwise equality mechanism. Here \(T\) is a
triangle with centroid \(0\), \(d\) is the support distance from
\(0\) to a fixed side, and \(u_{\mathrm{side}}\) and
\(u_{\mathrm{vert}}=-u_{\mathrm{side}}\) are the unit normals pointing toward
that side and the opposite vertex. (a)~The direction
\(u_{\mathrm{side}}\) gives the increasing equality profile on
\([-2d,d]\), and \(u_{\mathrm{vert}}\) the decreasing one on \([-d,2d]\);
conversely, an affine profile vanishing at one endpoint forces both
boundary graphs to be affine, so \(K\) is a triangle. (b)~After the
normalization \(E_T=\Btwo\), the equality triangle is equilateral:
\(\Btwo\) is its incircle and \(2\Btwo\) its circumcircle, with the
colored arrows ending at the contact points.}
\label{fig:mechanism}
\end{figure}

In dimension \(n\), the optimal outer dilation factor in the
covariance-ellipsoid sandwich is \(n\), with simplices attaining
equality \cite{MP89}; see \cite{BGVV14} for background. Related sharp
trace second-moment inequalities, with simplices as the equality cases
among polytopes, are proved in \cite[Theorem~1.1]{FPS12}. The direct
planar proof above supplies the direction-by-direction equality
analysis needed in our compactness argument.

Figure~\ref{fig:mechanism} illustrates both halves of the equality
analysis. Its second panel uses the fact that every centroid-zero
triangle is a linear image of a centroid-zero equilateral one; by
linear equivariance of the covariance ellipse, the normalization
\(E_T=\Btwo\) therefore makes \(T\) equilateral up to an orthogonal
map.

\section{Compactness and the strict gap}

Let \(I_2\) denote the \(2\times2\) identity matrix. A body is in
\emph{isotropic position} if
\[
  c(K)=0,
  \qquad
  \Cov_K=\frac12I_2.
\]
Equivalently, \(c(K)=0\) and \(E_K=\Btwo\).
Every planar convex body has such an affine normalization, unique up
to an orthogonal map: explicitly, we take
\[
  K
  \longmapsto
  \frac1{\sqrt2}\Cov_K^{-1/2}(K-c(K)).
\]
We let \(\mathcal I\) be the class of bodies in isotropic position.

\begin{lemma}\label{lem:compact}
The class \(\mathcal I\) is compact in the Hausdorff topology, and
\((C,D)\mapsto\dcen(C,D)\) is continuous on
\(\mathcal I\times\mathcal I\).
\end{lemma}

\begin{proof}
Under the map above, Theorem~\ref{thm:sandwich} becomes
\[
  \Btwo\subseteq K\subseteq2\Btwo
  \qquad (K\in\mathcal I).
\]
Blaschke's selection theorem therefore gives sequential compactness
of the ambient family. Under Hausdorff convergence inside this fixed
sandwich, the indicator functions converge almost everywhere because
convex-body boundaries have area zero. Dominated convergence then
gives convergence of volumes and of the first and second moments.
Hence the conditions \(c(K)=0\) and
\(\Cov_K=\tfrac12I_2\) are closed, so \(\mathcal I\) is compact.

For continuity, let \(K,L\in\mathcal I\) and suppose
\(d_H(K,L)\le\eta\), where \(d_H\) denotes the Hausdorff distance.
Since \(\Btwo\subseteq K\) and \(\Btwo\subseteq L\), we have
\[
  K\subseteq L+\eta\Btwo\subseteq(1+\eta)L,
  \qquad
  L\subseteq K+\eta\Btwo\subseteq(1+\eta)K.
\]
Taking \(A=(1+\eta)I_2\) in the defining infimum for
\(\dcen(K,L)\) gives
\[
  K\subseteq AL\subseteq(1+\eta)^2K,
\]
and hence
\[
  \dcen(K,L)\le(1+\eta)^2.
\]

Now suppose \(K_j\to K\) and \(L_j\to L\) in the Hausdorff topology,
and put
\[
  \eta_j
  =
  \max\{d_H(K_j,K),d_H(L_j,L)\}.
\]
The multiplicative triangle inequality, applied in both directions,
then gives
\[
  (1+\eta_j)^{-4}\dcen(K,L)
  \le\dcen(K_j,L_j)
  \le(1+\eta_j)^4\dcen(K,L),
\]
which proves continuity on \(\mathcal I\times\mathcal I\).
\end{proof}

\begin{proof}[Proof of Theorem~\ref{thm:main}]
Independent affine normalizations do not change \(\dcen\).
Consequently, the supremum of \(\dcen\) over all pairs equals its
supremum over \(\mathcal I\times\mathcal I\). By
Lemma~\ref{lem:compact}, the latter supremum is attained; call it
\(M\).

For \(K\in\mathcal I\), put
\[
  q(K)=\max_{u\in S^1}h_K(u).
\]
Theorem~\ref{thm:sandwich} gives \(q(K)\le2\). Since
\(h_{\Btwo}\equiv1\), the support-function containment criterion used
in the proof of Theorem~\ref{thm:sandwich} gives
\[
  \Btwo\subseteq K\subseteq q(K)\Btwo.
\]
Therefore, for \(C,D\in\mathcal I\), we have
\[
  C\subseteq q(C)\Btwo\subseteq q(C)D
  \subseteq q(C)q(D)\Btwo\subseteq q(C)q(D)C,
\]
and hence
\[
  \dcen(C,D)\le q(C)q(D)\le4.
\]
Thus \(M\le4\).

Suppose that \(M=4\), and choose \(C,D\in\mathcal I\) with
\(\dcen(C,D)=4\). The preceding inequalities then force
\(q(C)=q(D)=2\). The maxima defining \(q(C)\) and \(q(D)\) are
attained, so the equality statement in Theorem~\ref{thm:sandwich}
implies that both bodies are triangles.

Write \(C=\operatorname{conv}\{c_1,c_2,c_3\}\) and
\(D=\operatorname{conv}\{d_1,d_2,d_3\}\), where
\(c_1+c_2+c_3=d_1+d_2+d_3=0\).
Because the triangles are nondegenerate, \(c_1,c_2\) and \(d_1,d_2\)
are bases of \(\Rtwo\). The linear map taking \(d_i\) to \(c_i\) for
\(i=1,2\) also takes \(d_3=-d_1-d_2\) to \(c_3=-c_1-c_2\), and hence
maps \(D\) onto \(C\). Thus \(\dcen(C,D)=1\), contradicting the
choice of the pair.

Thus \(M<4\). Taking \(\varepsilon_0=4-M\) completes the proof.
\end{proof}

\providecommand{\bysame}{\leavevmode\hbox to3em{\hrulefill}\thinspace}

\end{document}